\documentclass[preprint,11pt]{elsarticle}
\makeatletter
\def\ps@pprintTitle{%
 \let\@oddhead\@empty
 \let\@evenhead\@empty
 \def\@oddfoot{\centerline{\thepage}}%
 \let\@evenfoot\@oddfoot}
\makeatother
\usepackage{hyperref}
\usepackage{lineno}
\usepackage{isomath}
\usepackage{bm}
\usepackage{setspace}
\usepackage{url}
\usepackage{breakurl}

\usepackage{soul}
\usepackage{graphics}
\usepackage{psfig}
\usepackage{rotating}
\usepackage[absolute]{textpos}
\usepackage{color}
\usepackage{lineno}
\usepackage[utf8]{inputenc}
\usepackage{amsmath}
\usepackage{amsfonts}
\usepackage{amssymb}
\usepackage{amsthm}
\usepackage{mathrsfs}
\usepackage{latexsym}
\usepackage{multirow}
\usepackage{rotating}
\usepackage{caption}
\usepackage{graphicx}
\usepackage{float}
\usepackage{array}
\usepackage{subfigure}
\usepackage{tabularx}
\usepackage{graphicx,psfig}
\usepackage{footnote}
\usepackage[titletoc,toc,title]{appendix}
\usepackage{colortbl}
\usepackage[dvipsnames]{xcolor}
\usepackage{hyperref}
\usepackage{natbib}
\usepackage{amssymb}
\usepackage{soul}
\usepackage{algorithmic}
\usepackage{algorithm}
\usepackage{epsfig}
\usepackage{relsize}
\usepackage{color}
\usepackage{mathtools}
\usepackage{multirow}
\usepackage[utf8]{inputenc}
\usepackage[english]{babel}

\newtheorem{theorem}{Theorem}
\newtheorem{lemma}{Lemma}

\newtheorem{prop}{Proposition}
\newtheorem*{remark}{Remark}

\allowdisplaybreaks
 \journal{Statistics \& Probability Letters}

\begin{document}
%\doublespacing
\begin{frontmatter}

\title{A nonparametric ensemble binary classifier and its statistical properties}
%\author{Tanujit Chakraborty, Ashis Kumar Chakraborty, C.A. Murthy}
\author{Tanujit Chakraborty\footnote[1]{\textit{Corresponding author}:
Tanujit Chakraborty (tanujit\_r@isical.ac.in)}, Ashis Kumar Chakraborty\textsuperscript{2} and C.A.
Murthy\textsuperscript{3}\\
    {\scriptsize \textsuperscript{1 and 2} SQC and OR Unit, Indian Statistical Institute, 203, B. T. Road, Kolkata - 700108, India}\\
    {\scriptsize \textsuperscript{3}Machine Intelligence Unit, Indian Statistical Institute, 203, B. T. Road, Kolkata - 700108, India}}
\begin{abstract}
%\doublespacing
In this work, we propose an ensemble of classification trees (CT)
and artificial neural networks (ANN). Several statistical properties
including universal consistency and upper bound of an important
parameter of the proposed classifier are shown. Numerical evidence
is also provided using various real life data sets to assess the
performance of the model. Our proposed nonparametric ensemble
classifier doesn't suffer from the ``curse of dimensionality'' and
can be used in a wide variety of feature selection cum
classification problems. Performance of the proposed model is quite
better when compared to many other state-of-the-art models used for
similar situations.
\end{abstract}

\begin{keyword}
Classification trees, feature selection, neural networks, hybrid
model, consistency, medical data sets.
\end{keyword}
\end{frontmatter}

\section{Introduction} \label{Introduction}
%\doublespacing
Distribution-free classification models are specially used in the
fields of statistics and machine learning since more than forty
years now, mainly for their accuracy and ability to deal with high
dimensional features and complex data structures. Two such models
are CT and ANN. Both have the ability to model arbitrary decision
boundaries. CT is found to be robust when limited data are available
unlike ANN. But decision trees are high variance estimators and
greedy in nature \cite{breiman2017classification} whereas neural
networks are universal approximators \cite{hornik1989multilayer}.
More powerful ANN has many free tuning parameters and risk of
over-fitting for small data sets. To utilize the positive aspects of
these two powerful models, theoretical frameworks for combining both
classifiers are often used jointly to make decisions. The ultimate
goal of designing classification models is to achieve best possible
performance in terms of accuracy measures for the task at hand. This
objective led researchers to create efficient hybrid models and
prove their statistical properties to make their best use. Mapping
tree based models to neural networks allows exploiting the former to
initialize the latter. Parameter optimization within ANN framework
will yield a model that is intermediate between CT and ANN as found
in some literatures \cite{sethi1990entropy, sakar1993growing}. Tsai
neural tree model \cite{tsai2012decision} uses the idea of splitting
the parameter space into areas by CT and builds in each of the areas
a locally specialized ANN model \cite{sirat1990neural}. In deep
neural tree model \cite{yang2018deep}, a decision tree provides the
final prediction and it differs from conventional CT by introducing
a global optimization of split and leaf node parameters using ANN.
But the major disadvantages of these algorithms are slow training,
having many tuning parameters, easy sticking on local minima and
poor robustness \cite{tsai2012decision}. All these hybrid models are
empirically shown to be useful in solving real life problems, but
the theoretical results are yet to be proved for many of them.

On the theoretical side, the literature is less conclusive, and
regardless of their use in practical problems of classification,
little is known about the statistical properties of these models.
The most celebrated theoretical result has given the general
sufficient conditions for almost-sure $L_{1}$-consistency of
histogram-based classification and data-driven density estimates
\cite{lugosi1996consistency}. In case of neural networks, it is
theoretically proven that if a one hidden layered (1HL) neural
network is trained with appropriately chosen number of neurons to
minimize the empirical risk on the training data, then it results in
a universally consistent classifier \cite{farago1993strong,
lugosi1995nonparametric}. Devroye et al.
\cite{devroye2013probabilistic} have theoretically shown that there
is some gain in considering two hidden layers (2HL), but it is not
really necessary to go beyond 2HL in ANN. In case of hybrid models,
the asymptotic results are less explored in the literature other
than a few notable works on neural decision trees
\cite{balestriero2017neural} and neural random forests
\cite{biau2016neural}. So there still exists a gap between theory
and practice when different hybrid models are considered.

Motivated by the above discussion, we have proposed in the present
paper an ensemble CT-ANN model which is an extension of our previous
work on hybrid CT-ANN model \cite{chakraborty2018novel}. Harnessing
the ensemble CT-ANN formulation, we try to exploit the strengths of
CT and ANN models and overcome their drawbacks. The approach is
mainly developed in theoretical details. Latter different training
schemes are experimentally evaluated on various small and medium
sized medical data sets having high dimensional feature spaces. The
model is found to be efficient for feature selection cum
classification task. We have established the consistency results and
upper bound for the model parameter of ensemble CT-ANN model which
assures a basic theoretical guarantee of efficiency of the proposed
algorithm. In our model, we have used CT as a feature selection
algorithm \cite{breiman2017classification} and have justified that
CT output plays an important role in further model building using
ANN algorithm. The proposed ensemble CT-ANN model has the advantages
of significant accuracy and very less number of tuning parameters.
Another major advantage of the proposed algorithm is its
interpretability as compared to more ``black-box-like" advanced
neural networks. Besides having the ability to deal with small and
medium sized data sets, our model is useful for selection of
important features and performing classification tasks in
high-dimensional feature spaces and complex data structures.

The paper is organized as follows. In section 2, we introduce
ensemble CT-ANN model. The main theoretical results are presented in
section 3 and application on various real life data sets are shown
in section 4. Section 5 is fully devoted to the conclusions and
future scope for research.

\section{Proposed Model} \label{Proposed_Model}
%\doublespacing
CT is a nonparametric classification algorithm which has a built-in
mechanism to perform feature selection \cite{quinlan1993c4}. Unlike
many other classification models, CT doesn't have any strong
assumption about normality of the data and homoscedasticity of the
noise terms. In our proposed model, we first split the feature space
into areas by CT algorithm. Most important features are chosen using
CT and redundant features are eliminated. Then we build ANN model
using the important variables obtained through CT algorithm along
with prediction results made by CT algorithm which is used as an
additional input feature in the neural networks. Then ANN model is
applied with one (hidden) layer to get the final classification
results. The optimum value of number of neurons in the hidden layer
is derived in Section 3. Since, we have taken CT output as an input
feature in ANN model, the number of hidden layer is chosen to be
one. We have also shown the proposed model to be universally
consistent in Section 3. The effectiveness of ensemble CT-ANN model
lies in the selection of important features using CT model and also
incorporating CT predicted class levels as a feature in ANN model.
It is noted that the inclusion of CT output as an input feature
increases the dimensionality of feature space that results in better
class separability as well. The theoretical set-up is presented in
Section 3 which gives robustness and statistical aspects of the
proposed model. The informal work-flow of the proposed model is as
follows:
\begin{itemize}
\item  First, apply CT algorithm to train and build a decision tree and record important features.
\item  The prediction results of CT algorithm is also applied as an additional feature for further modeling.
\item  Using important input variables obtained from CT along with an additional input variable (CT output), a neural network is generated.
\item  Run one hidden-layered ANN algorithm with sigmoid activation function and record the classification
results.
\item  The optimum number of neurons in the hidden layer of the model to be chosen as $O\bigg(\sqrt{\frac{n}{d_{m}log(n)}}\bigg)$ [discussed in Section 3], where $n,d_{m}$ are number of training samples and number
of input features in ANN model, respectively.
\end{itemize}
Our proposed model can be used for feature selection cum complex
classification problems. On the theoretical side, it is necessary to
show the universal consistency of the proposed model and other
statistical properties for its robustness. We will now introduce a
set of regularity conditions to show the consistency of feature
selection algorithm (CT) and the role of CT output in the proposed
model. Finally consistency of the proposed model and optimal number
of neurons in hidden layer will be shown in Section 3. Analogous
model for regression problems and related results are addressed in
\cite{chakraborty2018a}. A flowchart of ensemble CT-ANN model is
presented in Figure 1.

\begin{figure}[t]
\centering
%\subfloat[]
%\subfig{
%\includegraphics[height=1cm,angle=0,width=.6\linewidth,]{mod_realworld.eps}%
  %\epsfig{file=amazon_sim_final.ps, height=2.5in, width=3in}
%\includegraphics[width=0.6\textwidth,natwidth=1250,natheight=1250]{mod_real_world.PNG}
\includegraphics[scale=0.40]{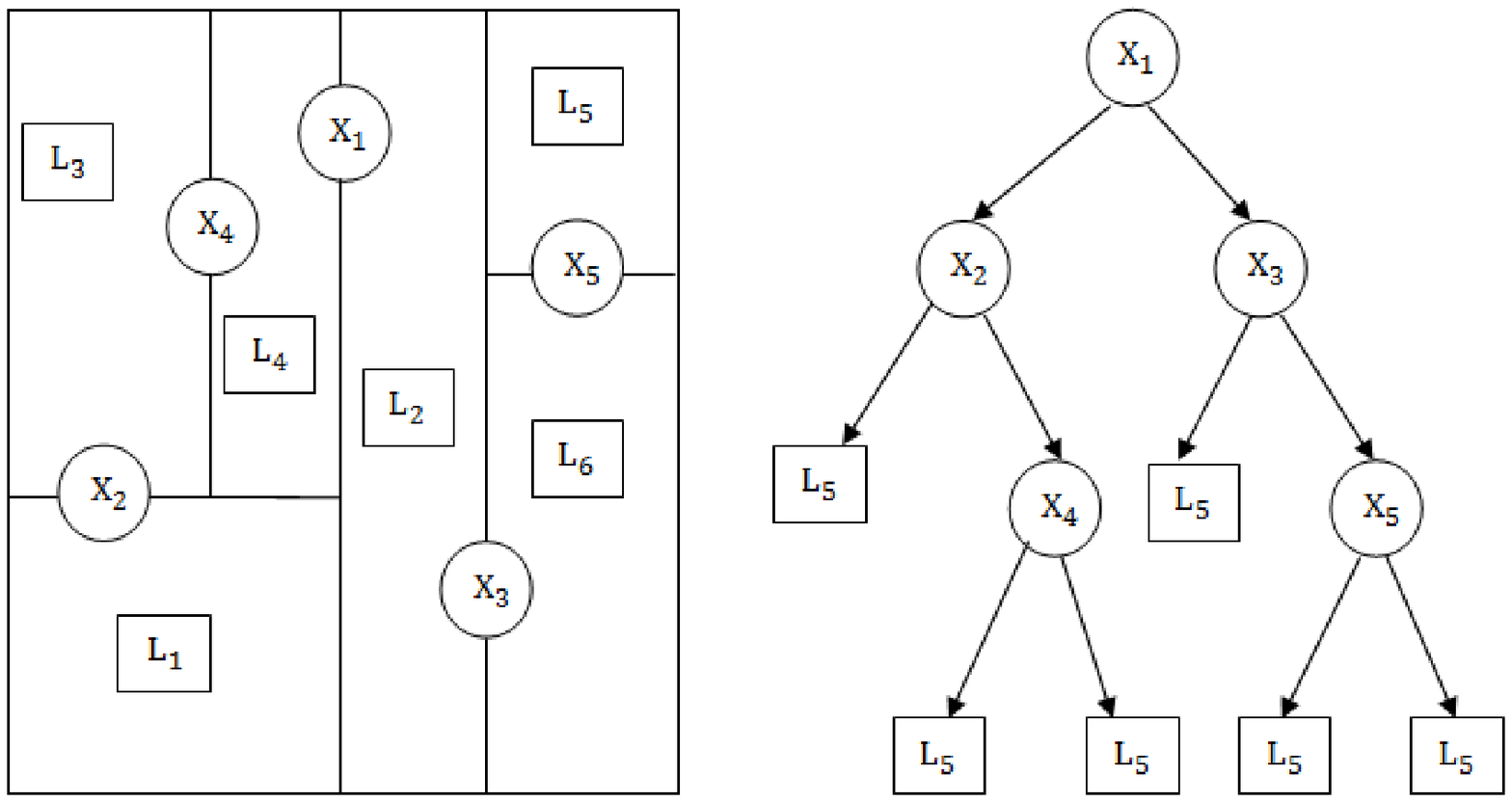}
\includegraphics[scale=0.40]{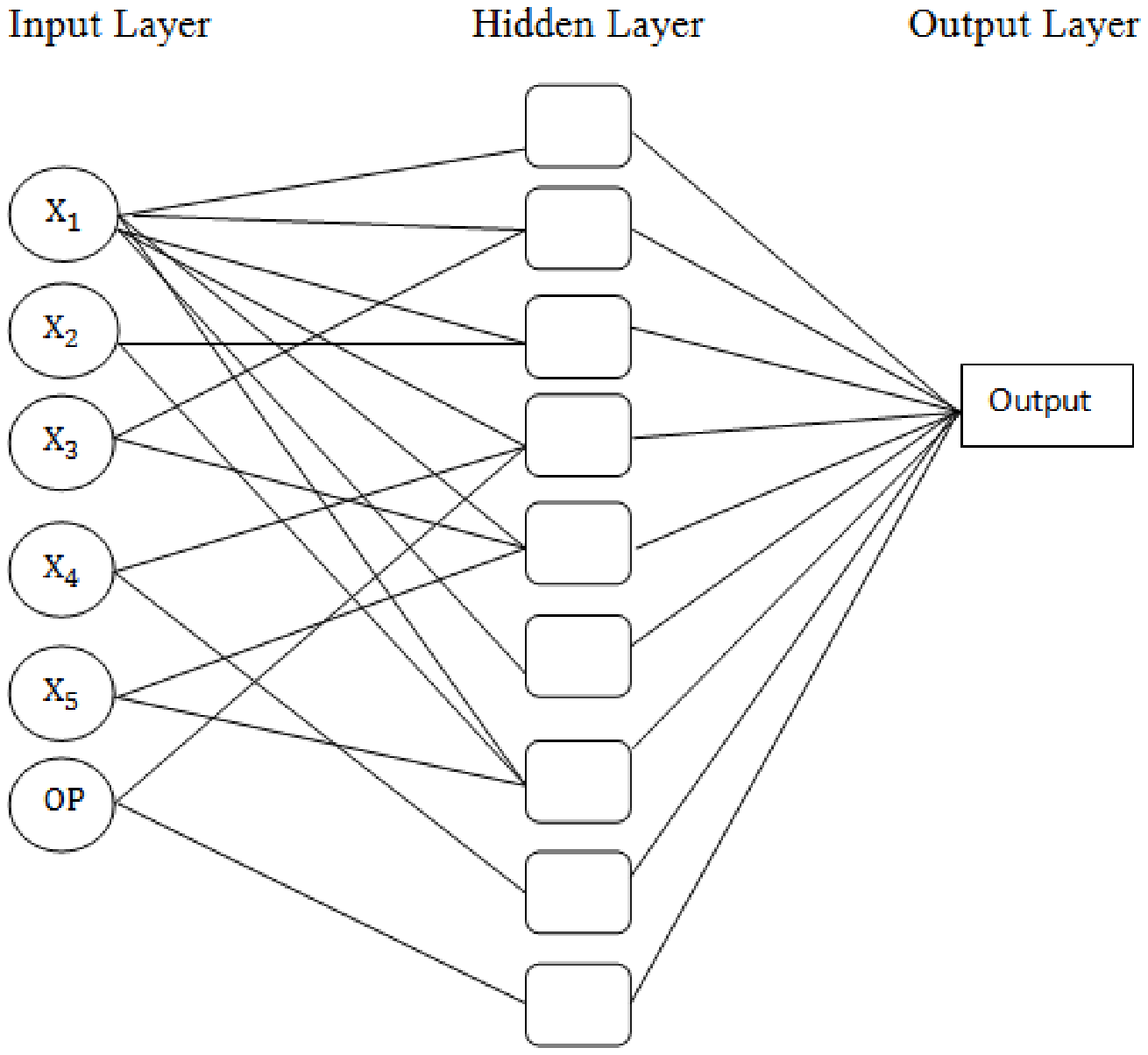}
\caption{An example of ensemble CT-ANN model with $X_{i}$, where
$i=1,2,3,4,5$, as important features obtained using CT, $L_{i}$ be
the leaf nodes and OP as CT output.} \label{figmodRWNGT}
\end{figure}

\section{Statistical Properties of the Proposed Model}  \label{Statistical Properties}
%\doublespacing
Our proposed ensemble classifier has the following nomenclature:
first, it extracts important features from the feature space using
CT algorithm, then it builds one hidden layered ANN model with the
important features extracted using CT along with CT outputs as an
additional feature. We investigate the statistical properties of the
proposed ensemble CT-ANN model by introducing a set of regularity
conditions for consistency of CT followed by the importance of CT
outputs for further model building. And finally we will discuss the
consistency results of ANN algorithm
with optimal value of number of neurons in the hidden layer of the model.\\

Let $\underline{X}$ be the space of all possible values of $p$
features and $C$ be the set of all possible binary outcomes. We are
given a training sample with $n$ observations $L=\{(X_{1},C_{1}),
(X_{2},C_{2}),...,(X_{n},C_{n})\}$, where
$X_{i}=(X_{i1},X_{i2},...,X_{ip}) \in \underline{X}$ and $C_{i} \in
C$. Also let $\Omega=\{\omega_{1},\omega_{2},...,\omega_{k}\}$ be a
partition of the feature space $\underline{X}$. We denote
$\widetilde{\Omega}$ as one such partition of $\Omega$. Define
$L_{\omega_{i}}=\{(X_{i},C_{i})\in L: X_{i}\in \omega_{i}, C_{i}\in
C\}$ as the subset of $L$ induced by $\omega_{i}$ and let
$L_{\widetilde{\Omega}}$ denote the partition of $L$ induced by
$\widetilde{\Omega}$. The information gain (to be introduced later)
from the feature space is used to partition $\underline{X}$ into a
set $\widetilde{\Omega}$ of nodes and then we can construct a
classification function on $\widetilde{\Omega}$. There exists a
partitioning classification function $d:\widetilde{\Omega}
\rightarrow C$ such that $d$ is constant on every node of
$\widetilde{\Omega}$. Now, let us define $\widehat{L}$ to be the
space of all learning samples and $\mathbb{D}$ be the space of all
partitioning classification function, then $\Phi:\widehat{L}
\rightarrow \mathbb{D}$ such that $\Phi(L)=(\psi \circ \phi)(L)$,
where $\phi$ maps $L$ to some induced partition
$(L)_{\widetilde{\Omega}}$ and $\psi$ is an assigning rule which
maps $(L)_{\widetilde{\Omega}}$ to $d$ on the partition
$\widetilde{\Omega}$. The most basic reasonable assigning rule
$\psi$ is the plurality rule $\psi_{pl}(L_{\widetilde{\Omega}})=d$
such that if $x \in \omega_{i}$, then
\[
d(\underline{x})=\arg \max_{c \in C}|L_{c,\omega_{i}}|
\]
The plurality rule is used to classify each new point in
$\omega_{i}$ as belonging to the class (either 0 or 1 in this case)
most common in $L_{\omega_{i}}$. This rule is very important in
proving risk consistency of the CT algorithm. Now, let us define a
binary split function $s(\omega_{i})$, that maps one node to a pair
of nodes, i.e.,
$s(\omega_{i})=(s_{1}(\omega_{i}),s_{2}(\omega_{i}))=(\omega_{1},\omega_{2})$,
then $\omega_{1}\cup \omega_{2}=\omega_{i}$, $\omega_{1}\cap
\omega_{2}=\phi$ and $\omega_{1},\omega_{2}\neq \phi$. Binary split
function partitions a parent node $\omega_{i} \subseteq
\underline{X}$ into a non-empty child nodes $\omega_{1}$ and
$\omega_{2}$, called left child and right child node respectively.
A set of all potential rules that we will use to split $\underline{X}$ is a finite set $S=\{ s_{1},s_{2},....s_{m}\}$.\\
Define $\mathscr{G}$ as a goodness of split criterion which maps
$(\omega_{i},s)$ for all $\omega_{i} \in \underline{X}$ and $s \in
S$ to a real number. For any parent node $\omega_{i}$, the goodness
of split criterion ranks the split functions. We have used the
following impurity function is used as goodness of split criterion:
\[
\mathscr{G}(L_{\omega_{i}},s)=H(L_{\omega_{i}})-\frac{|L_{s_{1}(\omega_{i})}|}{|L_{\omega_{i}}|}H(L_{\omega_{1}(t)})-\frac{|L_{s_{2}(\omega_{i})}|}{|L_{\omega_{i}}|}H(L_{\omega_{2}(t)})
\]
where, $H$ is taken as Gini Index and can be written as follows:
\[
H_{gini}(\omega_{i})=\sum_{c\neq
c'}\frac{|L_{\omega_{i},c}|}{|L_{\omega_{i}}|}.\frac{|L_{\omega_{i},
c'}|}{|L_{\omega_{i}}|}
\]
This criterion assesses the quality of a split $s(\omega_{i})$ by
subtracting the average impurity of the child nodes
$\omega_{1},\omega_{2}$ from the impurity of the parent node
$\omega_{i}$. The stopping rule in CT is decided based on the
minimum number of split in the posterior sample called minsplit
($r(\omega_{i})$). If $r(\omega_{i})$ $\geq \alpha$, then
$\omega_{i}$ will split into two child nodes and if $r(\omega_{i})$
$< \alpha$, then $\omega_{i}$ is a leaf node and no more split is
required. Here $\alpha$ is determined by the user, though for
practice it is usually taken as 10\% of the training sample size.

A binary tree-based classification and partitioning scheme $\Phi$ is
defined as an assignment rule applied to the limit of a sequence of
induced partitions $\phi^{(i)}(L)$, where $\phi^{(i)}(L)$ is the
partition of the training sample $L$ induced by the partition
$(\phi_{i} \circ \phi_{i-1} \circ .... \circ
\phi_{1})(\underline{X})$. For every node $\omega_{i}$ in a
partition $\widetilde{\Omega}$ such that $r(\omega_{i})$ $\geq
\alpha$, then the function $\phi(\widetilde{\Omega})$ splits each
node into two child nodes using the binary split in the question set
as determined by $\mathscr{G}$. For other nodes $\omega_{i} \in
\widetilde{\Omega}$ such that $r(\omega_{i})$ $< \alpha$, then
$\phi(\widetilde{\Omega})$ leaves $\omega_{i}$ unchanged.
Mathematically,
\begin{eqnarray}
\Phi(L)=(\psi \circ \lim_{i\rightarrow \infty}\phi^{(i)})(L)
\end{eqnarray}
where, $\phi^{(i)}(L)=L_{(\phi_{i} \circ \phi_{i-1} \circ .... \circ
\phi_{1})(\underline{X})}$. \\

CT as an axis-parallel split on $\mathbb{R}^{p}$ splits a node by
dividing into child nodes consisting of either side of some
hyperplane. We need to show that CT scheme are well defined, which
will be possible only if there exists some induced partition $L^{'}$
such that $\lim_{i \rightarrow \infty} \phi^{(i)}(L)=L^{'}$. In fact
we need to show that the following lemma holds:

\begin{lemma}
If L is a training sample and $\phi^{(i)}$ is defined as above, then
there exists $N \in \mathbb{N} \quad \mbox{such that for} \quad n
\geq N$
\begin{eqnarray}
\phi^{(n)}(L)=\lim_{i \rightarrow\infty} \phi^{(i)}(L)
\end{eqnarray}
\end{lemma}
\begin{proof}
Let $\{ L_{\widetilde{\Omega}} \}$ denote the sequence $\{  L,
\phi^{1}(L), \phi^{2}(L), ...  \}$. Defining
$\omega_{i}^{max}=\max{\{\omega_{i} \in \widetilde{\Omega}_{i}:
r(\omega_{i})
> \alpha \}}$ as the size of the largest non-leaf node(s) in
$\widetilde{\Omega}_{i}$. Suppose there exists $N \in \mathbb{N}$
such that $(\omega_{i})_{N}^{max}$ does not exist. Then every node
in $\widetilde{\Omega}_{N}$ is leaf. For all $n > N$,
$\widetilde{\Omega}_{n}=\widetilde{\Omega}_{N}$, then $(2)$ holds.
The sequence $\{ |\omega_{i}^{max}| \}$ is strictly decreasing if it
exists. Further if $\omega_{i+1}^{max}$ exists then
$|\omega_{i+1}^{max}| \leq |\omega_{i}^{max}| - 1$ and
$|\omega_{i}^{max}| \geq 1$ and $|\omega_{1}^{max}|=|L|$. This means
that $(\omega_{i})_{|L|}^{max}$ can not exist, so $(2)$ always holds
with $N \leq |L|$.
\end{proof}

For a wide range of partitioning schemes, the consistency of
histogram classification schemes based on data-dependent partitions
was shown in the literature \cite{lugosi1996consistency}.
To introduce the theorem, we need to define the following:\\

For any random variable X and set A, let
$\eta_{n,X}(A)=\frac{1}{n}\sum_{i=1}^{n}I(X_{i} \in A)$ be the
empirical probability that $X \in A$ based on $n$ observations and
$I(z)$ denotes the indicator of an event C. Now let $\mathcal{T}
=(\widetilde{\Omega}_{1},\widetilde{\Omega}_{2},...)$ be a finite
collection of partitions of a measurement space $\underline{X}$.
Define maximal node count of $\mathcal{T}$ as the maximum number of
nodes in any partition $\widetilde{\Omega}$ in $\mathcal{T}$ which
can be written as $\lambda(\mathcal{T})=\sup_{\widetilde{\Omega} \in
\mathcal{T}}|\widetilde{\Omega}|$. Also let, $\Delta (\mathcal{T},
L)=|\{ L_{\widetilde{\Omega}}: \widetilde{\Omega} \in \mathcal{T}|$
be the number of distinct partitions of a training sample of size
$n$ induced by partitions in $\mathcal{T}$. Let
$\Delta_{n}(\mathcal{T})$ be the growth function of $\mathcal{T}$
defined as $\Delta_{n}(\mathcal{T})=\sup_{\{L:|L|=n\}}
\Delta(\mathcal{T},L)$. Growth function of $\mathcal{T}$ is the
maximum number of distinct partitions $L_{\widetilde{\Omega}}$ which
partition $\widetilde{\Omega}$ in $\mathcal{T}$ can induce in any
training sample with $n$ observations. Take any class $\mathscr{A}
\subseteq \mathbb{R}^{p}$, $S_{n}(\mathscr{A})=\max_{\{B \subset
\mathbb{R}^{p}:|B|=n\}}|A \cap B : A \in \mathscr{A}|$ is the
maximum number of partitions of $B$ induced by $\mathscr{A}$, where
$B$ is some $n$ point subset of $\mathbb{R}^{p}$, called shatter
coefficient. For a partition $\widetilde{\Omega}$ of $X$, let
$\widetilde{\Omega}[x \in X]=\{ \omega \in \widetilde{\Omega}: x \in
\omega \}$ be the node $\omega$ in $\widetilde{\Omega}$ which
contains $x$. For a set $A \subseteq \mathbb{R}^{p}$, let
$D(A)=\sup_{x,y \in A}\parallel x-y \parallel$ be the diameter of
$A$. Now, for the sake of completeness we are rewriting the Theorem
2 of \cite{lugosi1996consistency} in our context:
\begin{theorem}
Let $(\underline{X},\underline{Y})$ be a random vector taking values
in $\mathbb{R}^{p} \times C$ and $L$ be the set of first n outcomes
of $(\underline{X},\underline{Y})$. Suppose that $\Phi$ is a
partition and classification scheme such that $\Phi(L)=(\psi_{pl}
\circ \phi)(L)$, where $\psi_{pl}$ is the plurality rule and
$\phi(L)=(L)_{\tilde{\Omega_{n}}}$ for some $\tilde{\Omega}_{n}\in
\mathcal{T}_{n}$, where
\[
\mathcal{T}_{n}=\{\phi(\ell_{n}): P(L=\ell_{n})>0\}.
\]
Also suppose that all the binary split functions in the question set
associated with $\Phi$ are hyperplane splits. As
$n\rightarrow\infty$, if the following regularity conditions hold:
\begin{equation}
\frac{\lambda(\mathcal{T}_{n})}{n}\rightarrow 0
\end{equation}
\begin{equation}
\frac{log(\triangle_{n}(\mathcal{T}_{n}))}{n}\rightarrow 0
\end{equation}
and for every $\gamma > 0$ and $\delta \in (0,1)$,
\begin{equation}
\inf_{S\subseteq \mathbb{R}^{p} : \eta_{x}(S) \geq 1-\delta}
\eta_{x}({x: diam(\tilde{\Omega}_{n}[x]\cap S) > \gamma})\rightarrow
0
\end{equation}
with probability 1. then $\Phi$ is risk consistent.
\end{theorem}
\begin{proof}
For the proof of Theorem 1, one may refer to
\cite{lugosi1996consistency}.
\end{proof}
\begin{remark}
Now instead of considering histogram-based partitioning and
classification schemes, we are going to show the risk consistency of
CT as defined above. We can produce a simultaneous result with
conditions (3) and (4) of Theorem 1 replaced by a simple condition.
Though the shrinking cell condition ((5) of Theorem 1) is also
assumed.
\end{remark}

\begin{theorem}
Suppose $(\underline{X},\underline{Y})$ be a random vector in
$\mathbb{R}^{p}\times\ C$ and $L$ be the training set consisting of
$n$ outcomes of $(\underline{X},\underline{Y})$. Let $\Phi$ be a
classification tree scheme such that
\[
\Phi(L)=(\psi_{pl}\circ \lim_{i\to\infty}\phi^{(i)})(L)
\]
where, $\psi_{pl}$ is the plurality rule and
$\phi(L)=(L)_{\tilde{\Omega}_{n}}$ for some $\tilde{\Omega}_{n}\in
\mathcal{T}_{n}$, where
\[
\mathcal{T}_{n}=\{\lim_{i\to\infty}\phi^{(i)}(\ell_{n}):
P(L=\ell_{n})>0\}.
\]
Suppose that all the split function in CT in the question set
associated with $\Phi$ are axis-parallel splits. Finally if for
every n and $w_{i} \in \tilde{\Omega}_{n}$, the induced subset
$L_{w_{i}}$ has cardinality $\geq k_{n}$, where
$\frac{k_{n}}{log(n))}\rightarrow \infty$ and (5) holds true, then
$\Phi$ is risk consistent.
\end{theorem}
\begin{proof}
Since $|w_{i}|\geq k_{n} \quad \forall \quad w_{i} \in
\tilde{\Omega}_{n}$, we can write
\begin{equation}
|\tilde{\Omega}_{n}|\leq \frac{n}{k_{n}}
\end{equation}
for every $\tilde{\Omega}_{n} \in \mathcal{T}_{n}$ and in that case,
we will have $\frac{\lambda(\mathcal{T}_{n})}{n} \leq
\frac{1}{k_{n}}$.\\
As $n\rightarrow \infty $, we can see $\frac{1}{k_{n}}\rightarrow 0$
which gives $\frac{\lambda(\mathcal{T}_{n})}{n}\rightarrow 0$. Hence
condition (3) holds true. \\

Now for every $\tilde{\Omega}_{n} \in \mathcal{T}_{n}$, using
Cover's theorem \cite{cover1965geometrical}, any binary split of
$\mathbb{R}^{p}$ can divide $n$ points in $\mathbb{R}^{p}$ in at
most $n^{p}$ ways. Using equation (6), we can write
\[
\Delta_{n}(\mathcal{T}_{n})\leq (n^{p})^{\frac{n}{k_{n}}}
\]
and consequently
\begin{equation}
\frac{log(\Delta_{n}(\mathcal{T}_{n}))}{n} \leq
p\frac{log(n)}{k_{n}}
\end{equation}
As $n\rightarrow \infty$, RHS of equation (7) goes to 0. So
condition (4) of Theorem 1 also holds and hence the theorem.
\end{proof}

\begin{remark}
Note that no assumptions are made on the distribution of the pair
$(\underline{X},\underline{Y}) \in \mathbb{R}^{p}\times C$. Also
sub-linear growth of the number of cells (condition (3)) and
sub-exponential growth of a combinatorial complexity measure
(condition (4)) are not required, instead a more flexible
restriction such as if each cell of $L_{\omega_{i}}$ has cardinality
$\geq k_{n}$ and $\frac{k_{n}}{log(n))}\rightarrow \infty$, then CT
is said to be risk consistent. So, feature selection using CT
algorithm is justified and now we are going to check the importance
of CT output in further model building. It is also noted that the
choice of important features based on CT is a greedy algorithm and
the optimality of local choices of the best feature for a node
doesn't guarantee that the constructed tree will be globally optimal
\cite{kuncheva2004combining}.
\end{remark}

Using CT given features, a list of important features are selected
from $p$ available features. It is noted that CT output also plays
an important role in further modeling. To see the importance of CT
given classification results (to be denoted by $OP$ in rest of the
paper) as a relevant feature, we introduce a non-linear measure of
correlation between any feature
and the actual class levels, namely C-correlation \cite{yu2004efficient}, as follows:\\

\textbf{Definition:} C-correlation is the correlation between any
feature $F_{i}$ and the actual class levels C, denoted by
$SU_{F_{i},C}$. Symmetrical uncertainty (SU)
\cite{press1992numerical} is defined as follows:
\begin{equation}
SU(X,Y)=2\bigg [\frac{H(X)-H(X|Y)}{H(X)+H(Y)} \bigg ]
\end{equation}
where, $H(X)$ is the entropy of a variable $X$ and $H(X|Y)$ is the
entropy of $X$ while $Y$ is observed. We can heuristically decide a
feature to be highly correlated with class $C$ if $SU_{F_{i},C} >
\beta$, where $\beta$ is a relevant threshold to be determined by
users. Using definition we can conclude that $OP$ can be taken as a
non-redundant feature for further model building. This also improves
the performance of the model at a
significant rate, to be shown in Section 4.\\

Now, we build ANN model with CT given features and $OP$ as another
input feature in ANN model. The dimension of input layer in ANN
model, denoted by $d_{m}(\leq{p})$, is the number of important
features obtained by CT + 1. We will use one hidden layer in ANN
model due to the incorporation of $OP$ as an input information in
the model. It should be noted that one-hidden layer neural networks
yield strong universal consistency and there is little theoretical
gain in considering two or more hidden layered neural networks
\cite{devroye2013probabilistic}. In ensemble CT-ANN model, we have
used one hidden layer with $k$ neurons. This makes the proposed
ensemble binary classifier less complex and less time consuming
while implementing the model. Our next objective is to state the
sufficient condition for universal consistency and then finding out
the optimal value of $k$. Before stating the sufficient conditions
for the consistency of the algorithm and optimal number of nodes in
hidden layer for practical use of the model, let us define the followings:\\

\textbf{Definition:} A sigmoid function $\sigma(x)$ is called a
logistic squasher if it is non-decreasing, $\lim_{x\rightarrow
\infty} \sigma(x)=0$ and $\lim_{x \rightarrow -\infty} \sigma(x)=1$
with $\sigma(x)=\frac{1}{1+exp{(-x)}}$.\\

Given $n$ training sequence
$\xi_{n}=\{(Z_{1},Y_{1}),...,(Z_{n},Y_{n})\}$ of $n$ i.i.d copies of
$(\underline{Z},\underline{Y})$ taking values from
$\mathbb{R}^{d_{m}}\times C$, a classification rule realized by a
one-hidden layered neural network is chosen to minimize the
empirical $L_{1}$ risk. Define the $L_{1}$ error of a function $\psi
: \mathbb{R}^{d_{m}}\rightarrow \{0,1\}$ by $J(\psi)=E\{ |\psi(Z)-Y|
\}$.

Consider a neural network with one hidden layer with bounded output
weight having $k$ hidden neurons and let $\sigma$ be a logistic
squasher. Let $\mathscr{F}_{n,k}$ be the class of neural networks
with logistic squasher defined as
\[
\mathscr{F}_{n,k}=\Bigg\{
\sum_{i=1}^{k}c_{i}\sigma(a_{i}^{T}z+b_{i})+c_{0} : k \in
\mathbb{N}, a_{i} \in \mathbb{R}^{d_{m}}, b_{i},c_{i} \in
\mathbb{R}, \sum_{i=0}^{k}|c_{i}|\leq \beta_{n} \Bigg\}
\]

Let $\psi_{n}$ be the function that minimizes the empirical $L_{1}$
error over $\psi_{n} \in \mathscr{F}_{n,k}$. Lugosi and Zeger (1995)
has shown that if $k$ and $\beta_{n}$ satisfy
\[
k \rightarrow \infty , \quad \beta_{n} \rightarrow \infty , \quad
\frac{k \beta_{n}^{4}log(k \beta_{n}^{2})}{n} \rightarrow 0
\]
and there exists $\delta (> 0)$ such that
$\frac{\beta_{n}^{4}}{n^{1-\delta}} \rightarrow 0$, then the
classification rule
\begin{equation}
  g_{n}(z)=\begin{cases}
    0, & \text{if $\psi_{n}(z)\leq 1/2$}.\\
    1, & \text{otherwise}.
  \end{cases}
\end{equation}
is strongly universally consistent \cite{lugosi1995nonparametric}.\\

Alternatively, $J(\psi_{n})-J^{*} \rightarrow 0$ in probability,
where $J(\psi_{n})=E\{|\psi_{n}(Z)-Y| | \xi_{n}\}$ and
$J^{*}=\inf_{\psi_{n}}J(\psi_{n})$ \cite{devroye2013probabilistic}.
Write
\[
J(\psi_{n})-J^{*} = \bigg(  J(\psi_{n})-\inf_{\psi \in
\mathscr{F}_{n,k}} J(\psi) \bigg)  + \bigg ( \inf_{\psi \in
\mathscr{F}_{n,k}} J(\psi) - J^{*} \bigg )
\]
where, $( J(\psi_{n})-\inf_{\psi \in \mathscr{F}_{n,k}} J(\psi) )$
is called estimation error and $( \inf_{\psi \in \mathscr{F}_{n,k}}
J(\psi) - J^{*} )$ is called approximation error. \\

Now, we are going to find out the optimal choice of $k$ using the
regularity conditions of strong universal convergence and the idea
of obtaining upper bounds on the rate of convergence, i.e., how fast
$J(\psi_{n})$ approaches to zero \cite{gyorfi2006distribution}. To
obtain an upper bound on the rate of convergence, we will have to
impose some regularity conditions on the posteriori probabilities.
Though in case of rate of convergence of estimation error, we will
have a distribution-free upper bound \cite{farago1993strong}. And to
obtain the optimal value of $k$, it is enough to find upper bounds
of the estimation and approximation errors. The upper bound of
approximation error investigated by Baron
\cite{barron1993universal}, is given in Lemma 2.

\begin{lemma}
Assume that there is a compact set $E \subset \mathbb{R}^{d_{m}}$
such that $Pr\{Z \in E\}=1$ and the Fourier transform
$\widetilde{P_{0}}(w)$  of $P_{0}(z)$ satisfies
\[
\int_{\mathbb{R}^{d_{m}}} |\omega||\widetilde{P_{0}}(\omega)|d\omega
< \infty
\]
then
\[
\inf_{\psi \in \mathscr{F}_{n,k}} E \bigg ( f(Z,\psi)-P_{0}(Z) \bigg
)^{2} \leq \frac{c}{k},
\]
where c is a constant depending on the distribution.
\end{lemma}

\begin{remark}
The previous condition on Fourier transformation and extensive
discussion on the properties of functions satisfying the condition
(including logistic squasher function) are given in the paper of
Baron \cite{barron1993universal}.
\end{remark}

\begin{prop}
For a fixed $d_{m}$, let $\psi_{n} \in \mathscr{F}_{c}$. The neural
network satisfying regularity conditions of strong universal
consistency and if the conditions of Lemma 2 holds, then the optimal
choice of $k$ is $O\bigg( \sqrt{\frac{n}{d_{m}log(n)}} \bigg)$.
\end{prop}
\begin{proof}
Applying Cauchy-Schwarz inequality in lemma 2, we can write
\[
\inf_{\psi \in \mathscr{F}_{n,k}} E\bigg | f(Z,\psi)-P_{0}(Z)  \bigg
| = O \bigg ( \frac{1}{\sqrt{k}} \bigg)
\]
It is well known \cite{devroyegy} that for any $\psi$
\[
J(\psi)-J^{*}\leq 2E\bigg | f(Z,\psi)-P_{0}(Z)  \bigg |
\]
So, the upper bound of approximation error is found to be $O \bigg (
\frac{1}{\sqrt{k}} \bigg)$. \\Though the approximation error goes to
zero as the number of neurons goes to infinity for strongly
universally consistent classifier. But in practical implementation
number of neurons is often fixed (eg., can't be increased with the
size of the training sample grows). Now, it is enough to bound the
estimation error.

Let us define $r(\psi_{n})= E (J(\psi_{n}))= P(\psi_{n}(Z)\neq Y)$
is the average error probability of $\psi_{n}$. Using lemma 3 of
\cite{farago1993strong}, we can write that the estimation error is
always $O \bigg( \sqrt{\frac{kd_{m}log(n)}{n}}\bigg)$. \\Bringing
the above facts together, we can write
\[
r(\psi_{n})-J^{*}=O \bigg( \sqrt{\frac{kd_{m}log(n)}{n}} +
\frac{1}{\sqrt{k}} \bigg)
\]
Now, to find optimal value of $k$, the problem reduces to equating
$\sqrt{\frac{kd_{m}log(n)}{n}}$ with $\frac{1}{\sqrt{k}}$, which
gives $k=O\bigg( \sqrt{\frac{n}{d_{m}log(n)}} \bigg)$.
\end{proof}

\begin{remark}
We can see a remarkable property of ensemble CT-ANN model from
Proposition 1. Since for this class of posteriori probability
function as shown in Lemma 2, the rate of convergence does not
necessarily depend on the dimension, in the sense that the exponent
is constant. Thus, it can be concluded that the proposed model will
not suffer from the curse of dimensionality.
\end{remark}

The optimal value of hidden nodes is found to be $O\bigg(
\sqrt{\frac{n}{d_{m}log(n)}} \bigg)$ in the ensemble CT-ANN model.
For practical use, if the data set is small, the recommendation is
to use $\bigg( \sqrt{\frac{n}{d_{m}log(n)}} \bigg)$ for achieving
utmost accuracy of the proposed model. Since the proposed ensemble
classifier possesses the desirable statistical properties, the
classifier can be called robust. The practical usefulness and
competitiveness of the proposed classifier in solving real life
feature selection cum classification problems will be shown in
Section 4.

\section{Experimental Evaluation}  \label{Experimental Evaluation}
%\doublespacing
\subsection{\textbf{The datasets}}

The proposed model is evaluated using six publicly available medical
data sets from Kaggle\footnote{https://www.kaggle.com/datasets} and
UCI Machine Learning
repository\footnote{https://archive.ics.uci.edu/ml/datasets.html}
dealing with various diseases like breast cancer, pima diabetes,
heart disease, promoter gene sequences, SPECT heart images, etc.
These binary classification data sets have limited number of
observations and high-dimensional feature spaces. Breast cancer data
set has 9 discrete features where as pima diabetes data set consists
of 8 continuous features in the input space
\cite{rodriguez2006rotation}. Heart disease data set originally
contained a total of 303 examples for 13 continuous features, out of
which 6 contained missing class values and 27 are disputed examples
which were removed from the data set. Promoter gene sequences data
set has 57 sequential DNA nucleotides attributes. SPECT images data
set is represented by 22 binary features that have either 0 or 1
values, but the data set is imbalanced in nature. Wisconsin breast
cancer data set consists of 699 examples carrying 9 continuous
features in the input space \cite{kurgan2001knowledge}. Table 1
gives a summary about these data sets.

\begin{table}[H]
\tiny \centering \caption{Characteristics of the data sets used in
experimental evaluation}
    \begin{tabular}{cccccc}
        \hline
        Dataset                  & Classes    & Objects & Number of      & Number of          & Number of          \\
                                 &            & $(n)$   & feature $(p)$  & $(+)$ve instances  & $(-)$ve instances  \\ \hline
        breast cancer            & 2          & 286     & 9              & 85                 & 201                \\
        heart disease            & 2          & 270     & 13             & 120                & 150                \\
        pima diabetes            & 2          & 768     & 8              & 500                & 268                \\
        promoter gene sequences  & 2          & 106     & 57             & 53                 & 53                 \\
        SPECT heart images       & 2          & 267     & 22             & 55                 & 212                \\
        wisconsin breast cancer  & 2          & 699     & 9              & 458                & 241                \\
        \hline
    \end{tabular}

\end{table}

\subsection{\textbf{Performance measures}}

The performance evaluation measures used in our experimental
analysis are based on the confusion matrix. Higher the value of
performance metrics, the better the classifier is. The
expressions for different performance measures as follows: \\
Classification Accuracy = $\frac{(TP+TN)}{(TP+TN+FP+FN)}$; F-measure
=$2\frac{\big(\mbox{Precision}\times\mbox{Recall}\big)}{\big(\mbox{Precision}+\mbox{Recall}\big)}$;\\
where, Precision = $\frac{TP}{TP+FP}$; Recall =
$\frac{TP}{TP+FN}$; and \\
TP (True Positive): correct positive prediction; FP (False
Positive): incorrect positive prediction; TN (True Negative):
correct negative prediction; FN (False Negative): incorrect negative
prediction.

\subsection{\textbf{Analysis of Results}}

In order to show the impact of the proposed 2-step pipeline model,
it is applied to the high-dimensional small or medium sized medical
data sets. These are such types of data sets in which not only
classification is the task but also feature selection plays a vital
role before it. We shuffled the observations in each of the 6
medical data sets randomly and split it into training, validation
and test data sets in a ration of 50 : 25 : 25. We have also
repeated each of the experiments 10 times with different randomly
assigned training, validation and testing data sets.

Our proposed algorithm is compared with Classification Tree (CT),
Random Forest (RF), Support Vector Machine (SVM), Artificial Neural
Network (ANN) with 1HL and 2HL, Entropy Nets, Tsai Neural Tree (NT)
\cite{tsai2012decision}, Deep Neural Decision Trees (DNDT)
\cite{yang2018deep} based on the different performance metrics. All
these classifiers are implemented in R Statistical software on a PC
with 2.1GHz processor with 8GB memory other than DNDT. We compared
the proposed model with 1-HL ANN and 2-HL ANN without employing
feature selection. Since the data sets are small and medium sized,
going beyond 2HL ANN will over-fit the data set
\cite{devroye2013probabilistic} and this is also reminiscent of
universal approximation theorem \cite{hornik1989multilayer}. For 1HL
ANN, number of hidden neurons used is $k \thickapprox \sqrt{n}$
\cite{devroye2013probabilistic} and for 2HL ANN, 2/3 of the input
sizes are taken as the number of neurons in the 1st HL and 1/3 of
the input sizes in case of 2nd HL \cite{zhou2002ensembling}. In a
similar way Tsai Neural Tree (NT) were also built and the accuracy
levels were compared. DNDT searches tree structure and parameter
with stochastic gradient descent which was implemented in TensorFlow
\cite{abadi2016tensorflow}, and it is a kind of GPU-accelerated
computing \cite{yang2018deep}. Breiman's random forest
\cite{breiman2001random} also has an in-built feature selection
mechanism which was implemented using \textit{party} implementation
in R and results are reported in Table 2.

To apply our proposed model to the medical data sets, we first apply
CT with minsplit (as defined in Section 3) as 10\% of the training
sample size using the \textit{rpart} package implementation in R. CT
model uses gini index of diversity with the available input feature
space. The variable importance indicator $C_{p}$ was used for
selection of variables to enter or leave CT model. Based on the
results of CT, important variables or features were chosen in the
final model along with CT output. The number of reduced features
after feature selection using CT are reported in Table 2. The number
of hidden neurons in the hidden layer is calculated using this
formula $k=\sqrt{\frac{n}{d_{m}log(n)}}$, where $n$ is the number of
training samples and $d_{m}$ as the number of input features in
neural networks. We have further normalized the data sets before
training the neural network. Min-max method is used for scaling the
data in an interval of $[0,1]$. For small or medium data sets, our
model recommends using the upper bound of the number of neurons in
the HL of the ensemble model. The ensemble CT-ANN model is trained
using \textit{neuralnet} implementation in R. Training time and
memory requirement for our proposed model is quite low as compared
with DNDT which needs availability of GPU. Table 2 gives the
obtained results from different classifiers used for experimental
evaluation over 6 medical data sets. We can conclude from Table 2
that the proposed model achieves overall highest accuracy while
working with reduced features as compared to other state-of-the-arts
for most of the data sets and remains competitive for other few data
sets as well.

\begin{table}[H]
\tiny \centering \caption{Results (and their standard deviation) of
classification algorithms over 6 medical data sets}
    \begin{tabular}{ccccc}
        \hline
        Classifiers           & Data set                  & The number of (reduced)   & Classification       & F-measure          \\
                              &                           & features after            & accuracy             &                    \\
                              &                           & feature selection         & (\%)                 &                    \\
                              \hline
        \multirow{6}{*}{CT}  & breast cancer              & 7                         & 68.26 (6.40)         & 0.70 (0.07)        \\
                              & heart disease             & 7                         & 76.50 (4.50)         & 0.81 (0.03)        \\
                              & pima diabetes             & 6                         & 71.85 (4.94)         & 0.74 (0.03)        \\
                              & promoter gene sequences   & 17                        & 69.43 (2.78)         & 0.73 (0.01)        \\
                              & SPECT heart images        & 9                         & 75.70 (1.56)         & 0.78 (0.00)        \\
                              & wisconsin breast cancer   & 8                         & 94.20 (2.98)         & 0.89 (0.01)        \\
                              \hline
        \multirow{6}{*}{RF}   & breast cancer             & 6                         & 69.00 (7.30)         & 0.72 (0.07)        \\
                              & heart disease             & 8                         & 80.19 (4.23)         & 0.84 (0.01)        \\
                              & pima diabetes             & 6                         & 73.49 (4.12)         & 0.76 (0.03)        \\
                              & promoter gene sequences   & 20                        & 71.26 (1.97)         & 0.75 (0.03)        \\
                              & SPECT heart images        & 10                        & 79.70 (1.23)         & 0.82 (0.01)        \\
                              & wisconsin breast cancer   & 8                         & 95.75 (2.01)         & 0.96 (0.02)        \\
                              \hline
        \multirow{6}{*}{SVM}  & breast cancer             & 9                         & 64.62 (5.21)         & 0.68 (0.05)        \\
                              & heart disease             & 13                        & 78.95 (4.89)         & 0.83 (0.01)        \\
                              & pima diabetes             & 8                         & 70.39 (3.56)         & 0.72 (0.03)        \\
                              & promoter gene sequences   & 57                        & 59.35 (1.37)         & 0.63 (0.02)        \\
                              & SPECT heart images        & 22                        & \textbf{83.46} (1.29)& \textbf{0.85} (0.00)\\
                              & wisconsin breast cancer   & 9                         & 93.30 (2.78)         & 0.94 (0.01)        \\
                              \hline
        \multirow{6}{*}{ANN (with 1HL)} & breast cancer   & 9                         & 61.58 (5.89)         & 0.64 (0.04)        \\
                              & heart disease             & 13                        & 73.56 (5.44)         & 0.79 (0.02)        \\
                              & pima diabetes             & 8                         & 66.78 (4.58)         & 0.69 (0.04)        \\
                              & promoter gene sequences   & 57                        & 61.77 (3.46)         & 0.65 (0.02)        \\
                              & SPECT heart images        & 22                        & 79.69 (0.23)         & 0.81 (0.01)        \\
                              & wisconsin breast cancer   & 9                         & 94.80 (2.01)         & 0.96 (0.01)        \\
                              \hline
        \multirow{6}{*}{ANN (with 2HL)} & breast cancer   & 9                         & 62.20 (5.12)         & 0.64 (0.03)        \\
                              & heart disease             & 13                        & 78.81 (3.96)         & 0.82 (0.03)       \\
                              & pima diabetes             & 8                         & 69.78 (3.89)         & 0.73 (0.02)        \\
                              & promoter gene sequences   & 57                        & 63.46 (2.19)         & 0.68 (0.02)        \\
                              & SPECT heart images        & 22                        & 82.71 (0.78)         & 0.84 (0.01)        \\
                              & wisconsin breast cancer   & 9                         & 95.60 (2.54)         & 0.96 (0.10)        \\
                              \hline
        \multirow{6}{*}{Entropy Nets} & breast cancer     & 7                         & 69.00 (6.25)         & 0.72 (0.05)        \\
                              & heart disease             & 7                         & 79.59 (4.78)         & 0.83 (0.01)        \\
                              & pima diabetes             & 6                         & 69.50 (4.05)         & 0.72 (0.02)        \\
                              & promoter gene sequences   & 17                        & 66.23 (1.98)         & 0.70 (0.01)        \\
                              & SPECT heart images        & 9                         & 76.64 (1.70)         & 0.78 (0.01)        \\
                              & wisconsin breast cancer   & 8                         & 95.96 (2.18)         & 0.96 (0.00)        \\
                              \hline
        \multirow{6}{*}{TSai NT} & breast cancer          & 7                         & 69.45 (7.17)         & 0.71 (0.07)        \\
                              & heart disease             & 7                         & 80.25 (4.68)         & 0.85 (0.01)        \\
                              & pima diabetes             & 6                         & 71.59 (4.19)         & 0.74 (0.03)        \\
                              & promoter gene sequences   & 17                        & 70.67 (2.83)         & 0.74 (0.02)        \\
                              & SPECT heart images        & 9                         & 76.95 (1.27)         & 0.78 (0.01)        \\
                              & wisconsin breast cancer   & 8                         & \textbf{97.40} (2.11)& \textbf{0.98} (0.01)\\
                              \hline
        \multirow{6}{*}{DNDT}& breast cancer              & 8                         & 66.12 (7.81)         & 0.68 (0.08)        \\
                              & heart disease             & 7                         & 81.05 (3.89)         & 0.86 (0.02)        \\
                              & pima diabetes             & 6                         & 69.21 (5.08)         & 0.72 (0.05)        \\
                              & promoter gene sequences   & 17                        & 69.06 (1.75)         & 0.71 (0.01)        \\
                              & SPECT heart images        & 10                        & 75.50 (0.89)         & 0.77 (0.00)        \\
                              & wisconsin breast cancer   & 7                         & 94.25 (2.14)         & 0.95 (0.00)        \\
                              \hline
        \multirow{6}{*}{Proposed Model} & breast cancer   & 7                         & \textbf{72.80} (6.54)& \textbf{0.77} (0.06)\\
                              & heart disease             & 7                         & \textbf{82.78} (4.78)& \textbf{0.89} (0.02)\\
                              & pima diabetes             & 6                         & \textbf{76.10} (4.45)& \textbf{0.79} (0.04)\\
                              & promoter gene sequences   & 17                        & \textbf{75.40} (1.50)& \textbf{0.79} (0.01)\\
                              & SPECT heart images        & 9                         & 81.03 (0.56)         & 0.82 (0.00)        \\
                              & wisconsin breast cancer   & 8                         & 97.30 (1.05)         & 0.98 (0.00)        \\
        \hline
    \end{tabular}
\end{table}

\section{Conclusion and Discussions} \label{conclusions}
%\doublespacing
In this paper, a novel nonparametric ensemble classifier is proposed
to achieve higher accuracy in classification performance with very
little computational cost (by working with a subset of input
features). Our proposed feature selection cum classification model
is robust in nature. Ensemble CT-ANN model is shown to be
universally consistent and less time consuming during the actual
implementation. We have also found the optimal value of the number
of neurons in the hidden layer so that the user will have less
tuning parameters to be controlled. The proposed model when applied
to real life data sets performed better compared to other
state-of-the-art models for most of the data sets and remained
competitive for the few other data sets. Situations when feature
selection is not a job in classification problems, our model may not
be too effective. But the ensemble classifier will have an edge
where the data analysis requires important variable selections in
the early stage followed by predictions using classifiers for
limited data sets. In the light of current advances in ANN, one
might ask a simple question : What is the need of a two-step
pipeline (like ensemble CT-ANN model) over advanced ANN models ?

A straight-cut answer to this question could be unwise. The primary
goal of 'statistics' is to make scientific inferences from the model
compared to building a ``black-box-like" model which may perform
well for some specific data sets, but may not be considered as a
general theory \cite{dunson2018statistics}. Our proposed model is
robust, universally consistent, easily interpretable and highly
useful for high dimensional small or medium sized data sets (for
example, medical data sets) to perform feature selection cum
classification task. Advanced ANN models (say, deep neural net) are
highly complex, over-parameterized models and found useful when the
data sets are very large (like image, audio and video data sets)
\cite{dunson2018statistics}. Nevertheless, no model can have
dominant advantage and one may also refer to \textit{no free lunch
theorems} \cite{wolpert1996lack}. Normally, for every new finding
there will always be a trade-off between accuracy, interpretability
and complexity of the model \cite{wolpert1996lack}. There are many
future scope of research of this paper. One scope is to extend the
model for multi-class classification problems. Another interesting
area for research is to improve the ensemble model especially for
imbalanced data sets.

\section*{Acknowledgements}
The authors are grateful to the editors and anonymous referees for
careful reading, constructive comments and insightful suggestions,
which have greatly improved the quality of the paper.

\bibliographystyle{elsarticle-num}
\bibliography{bibliography}

\begin{thebibliography}{10}
\expandafter\ifx\csname url\endcsname\relax
  \def\url#1{\texttt{#1}}\fi
\expandafter\ifx\csname urlprefix\endcsname\relax\def\urlprefix{URL }\fi
\expandafter\ifx\csname href\endcsname\relax
  \def\href#1#2{#2} \def\path#1{#1}\fi

\bibitem{breiman2017classification}
L.~Breiman, Classification and regression trees, Routledge, 2017.

\bibitem{hornik1989multilayer}
K.~Hornik, M.~Stinchcombe, H.~White, Multilayer feedforward networks are
  universal approximators, Neural networks 2~(5) (1989) 359--366.

\bibitem{sethi1990entropy}
I.~K. Sethi, Entropy nets: from decision trees to neural networks, Proceedings
  of the IEEE 78~(10) (1990) 1605--1613.

\bibitem{sakar1993growing}
A.~Sakar, R.~J. Mammone, Growing and pruning neural tree networks, IEEE
  Transactions on Computers 42~(3) (1993) 291--299.

\bibitem{tsai2012decision}
C.-C. Tsai, M.-C. Lu, C.-C. Wei, Decision tree-based classifier combined with
  neural-based predictor for water-stage forecasts in a river basin during
  typhoons: a case study in taiwan, Environmental engineering science 29~(2)
  (2012) 108--116.

\bibitem{sirat1990neural}
J.~Sirat, J.~Nadal, Neural trees: a new tool for classification, Network:
  Computation in Neural Systems 1~(4) (1990) 423--438.

\bibitem{yang2018deep}
Y.~Yang, I.~G. Morillo, T.~M. Hospedales, Deep neural decision trees, arXiv
  preprint arXiv:1806.06988.

\bibitem{lugosi1996consistency}
G.~Lugosi, A.~Nobel, Consistency of data-driven histogram methods for density
  estimation and classification, The Annals of Statistics 24~(2) (1996)
  687--706.

\bibitem{farago1993strong}
A.~Farag{\'o}, G.~Lugosi, Strong universal consistency of neural network
  classifiers, IEEE Transactions on Information Theory 39~(4) (1993)
  1146--1151.

\bibitem{lugosi1995nonparametric}
G.~Lugosi, K.~Zeger, Nonparametric estimation via empirical risk minimization,
  IEEE Transactions on information theory 41~(3) (1995) 677--687.

\bibitem{devroye2013probabilistic}
L.~Devroye, L.~Gy{\"o}rfi, G.~Lugosi, A probabilistic theory of pattern
  recognition, Vol.~31, Springer Science \& Business Media, 2013.

\bibitem{balestriero2017neural}
R.~Balestriero, Neural decision trees, arXiv preprint arXiv:1702.07360.

\bibitem{biau2016neural}
G.~Biau, E.~Scornet, J.~Welbl, Neural random forests, Sankhya A (2016) 1--40.

\bibitem{chakraborty2018novel}
T.~Chakraborty, S.~Chattopadhyay, A.~K. Chakraborty, A novel hybridization of
  classification trees and artificial neural networks for selection of students
  in a business school, OPSEARCH 55~(2) (2018) 434--446.

\bibitem{quinlan1993c4}
J.~R. Quinlan, C4. 5: Programming for machine learning, Morgan Kauffmann 38
  (1993) 48.

\bibitem{chakraborty2018a}
T.~Chakraborty, A.~K. Chakraborty, S.~Chattopadhyay, A novel distribution-free
  hybrid regression model for manufacturing process efficiency improvement,
  arXiv preprint arXiv:1804.08698.

\bibitem{cover1965geometrical}
T.~M. Cover, Geometrical and statistical properties of systems of linear
  inequalities with applications in pattern recognition, IEEE transactions on
  electronic computers~(3) (1965) 326--334.

\bibitem{kuncheva2004combining}
L.~I. Kuncheva, Combining pattern classifiers: methods and algorithms, John
  Wiley \& Sons, 2004.

\bibitem{yu2004efficient}
L.~Yu, H.~Liu, Efficient feature selection via analysis of relevance and
  redundancy, Journal of machine learning research 5~(Oct) (2004) 1205--1224.

\bibitem{press1992numerical}
W.~H. Press, S.~A. Teukolsky, W.~T. Vetterling, B.~P. Flannery, Numerical
  recipes in c: the art of scientific computing, Cambridge University Press,
  Cambridge, MA, 131 (1992) 243--262.

\bibitem{gyorfi2006distribution}
L.~Gy{\"o}rfi, M.~Kohler, A.~Krzyzak, H.~Walk, A distribution-free theory of
  nonparametric regression, Springer Science \& Business Media, 2006.

\bibitem{barron1993universal}
A.~R. Barron, Universal approximation bounds for superpositions of a sigmoidal
  function, IEEE Transactions on Information theory 39~(3) (1993) 930--945.

\bibitem{devroyegy}
L.~Devroye, L.~Gyorfi, Nonparametric density estimation: The l1 view, New York
  : John Wiley \& Sons, 1985.

\bibitem{rodriguez2006rotation}
J.~J. Rodriguez, L.~I. Kuncheva, C.~J. Alonso, Rotation forest: A new
  classifier ensemble method, IEEE transactions on pattern analysis and machine
  intelligence 28~(10) (2006) 1619--1630.

\bibitem{kurgan2001knowledge}
L.~A. Kurgan, K.~J. Cios, R.~Tadeusiewicz, M.~Ogiela, L.~S. Goodenday,
  Knowledge discovery approach to automated cardiac spect diagnosis, Artificial
  intelligence in medicine 23~(2) (2001) 149--169.

\bibitem{zhou2002ensembling}
Z.-H. Zhou, J.~Wu, W.~Tang, Ensembling neural networks: many could be better
  than all, Artificial intelligence 137~(1-2) (2002) 239--263.

\bibitem{abadi2016tensorflow}
M.~Abadi, P.~Barham, J.~Chen, Z.~Chen, A.~Davis, J.~Dean, M.~Devin,
  S.~Ghemawat, G.~Irving, M.~Isard, et~al., Tensorflow: a system for
  large-scale machine learning., in: OSDI, Vol.~16, 2016, pp. 265--283.

\bibitem{breiman2001random}
L.~Breiman, Random forests, Machine learning 45~(1) (2001) 5--32.

\bibitem{dunson2018statistics}
D.~B. Dunson, Statistics in the big data era: Failures of the machine,
  Statistics \& Probability Letters 136 (2018) 4--9.

\bibitem{wolpert1996lack}
D.~H. Wolpert, The lack of a priori distinctions between learning algorithms,
  Neural computation 8~(7) (1996) 1341--1390.

\end{thebibliography}

\end{document}